\documentclass[11pt, a4paper]{article}
\usepackage{t1enc}
\usepackage[latin1]{inputenc}
\usepackage[english]{babel}
\usepackage{amsmath,amsthm}
\numberwithin{equation}{section}
\usepackage{amsfonts}
\usepackage{latexsym}
\usepackage{graphicx}
\usepackage{float}
\usepackage[natural]{xcolor}
\usepackage{algorithm}
\usepackage{algorithmic}
\usepackage{enumerate,enumitem}
\usepackage{multirow}
\usepackage[colorlinks,linkcolor=blue]{hyperref}
\usepackage{lineno}
\usepackage{setspace}
\usepackage{fullpage}
\usepackage[title]{appendix}
\usepackage{amssymb}
\usepackage{listings}
\usepackage{bbm}

\usepackage[margin=2.3cm]{geometry}
\usepackage{enumitem,verbatim}

\newtheorem{theorem}{Theorem}[section]
\newtheorem*{thm-non}{Theorem}

\newtheorem{proposition}[theorem]{Proposition}
\newtheorem{lemma}[theorem]{Lemma}
\newtheorem{corollary}[theorem]{Corollary}

\newtheorem{problem}[theorem]{Problem}
\theoremstyle{definition}

\newtheorem*{defn-non}{Definition}

\title{%
Sim-Width, Induced Matching Treewidth, and Tree-Independence Number in Induced $K_{t,t}$-Free Graphs
\vspace{-6pt}
}

\author{%
Mengyuan Niu\textsuperscript{1,2},\quad
Xiumei Wang\textsuperscript{1}\thanks{%
Corresponding author: Xiumei Wang.
E-mail: wangxiumei@zzu.edu.cn.}\\[10pt]
{\small \textsuperscript{1}School of Mathematics and Statistics,
Zhengzhou University, Zhengzhou 450001, China}\\
{\small
\textsuperscript{2}Institute for Basic Science, 55 Expo-ro,
Yuseong-gu, Daejeon 34126, Republic of Korea}
}
\date{}\makeatother

\begin{document}
\maketitle
\begin{abstract}
The tree-independence number $tree\text{-}\alpha(G)$, the induced matching treewidth $tree\text{-}\mu(G)$, and the sim-width $simw(G)$ are graph parameters defined in terms of tree or branch decompositions. We establish two polynomial bounds for the tree-independence number of induced $K_{t,t}$-free graphs, one in terms of sim-width and the other in terms of induced matching treewidth.

Abrishami et al. (SIDMA, 2025) and Brettell et al. (EJC, 2025) asked whether bounded sim-width, together with the exclusion of an induced $K_{t,t}$, implies bounded tree-independence number. 
We answer this question by proving that, for integers $t\geq 2$ and $s\geq 1$, every induced $K_{t,t}$-free graph $G$ with $simw(G)\leq s$ satisfies $tree\text{-}\alpha(G)=O_t\left((s+1)^{2t^2-2t}\right)$. 
This also proves a polynomial strengthening of a conjecture of Be\v{s}ter \v{S}torgel et al. (arXiv, 2026)
concerning induced $K_{1,t}$-free graphs and improves a theorem of Alon et al. (arXiv, 2025) by reducing the exponent from $3t^2+1$ to $2t^2-2t$.

Alon et al. (arXiv, 2025) asked whether, for fixed induced matching treewidth, the tree-independence number is polynomially bounded in $t$. Using a VC-dimension argument, we answer this question affirmatively by showing that, for integers $\mu\geq 1$ and $t\geq 2$, every induced $K_{t,t}$-free graph $G$ with $tree\text{-}\mu(G)\leq\mu$ satisfies $tree\text{-}\alpha(G)=t^{O_\mu(1)}$.
\end{abstract}
\maketitle
\par\medskip

\section{Introduction}\label{sec:intro}
For standard notation and
terminology, we refer the reader to Section~\ref{s2}.
Tree decompositions form one of the fundamental local-to-global frameworks in structural and algorithmic graph theory. The notion now known as treewidth has roots in several lines of research, including nonserial dynamic programming \cite{BB72}, Halin's $S$-functions \cite{Halin76}, the Graph Minors project \cite{RS84}, and partial $k$-trees \cite{ACP87}. Robertson and Seymour \cite{RS84} were the first to use the term treewidth.
A tree decomposition represents a graph by overlapping vertex sets, called bags, organized in a tree structure, so that local information on the bags can yield global properties of the graph.
The treewidth of a graph  is the minimum, taken over all its tree decompositions, of the size of the largest bag minus one.
It plays a fundamental role in graph minor theory \cite{Lovasz06}, decomposition-based algorithms \cite{ALS91,CyganEtAl15}, and logical meta-theorems \cite{Courcelle90}.
Its applicability to dense graphs is inherently limited, as bounded treewidth imposes sparsity properties; for example, $tw(K_n)=n-1$.

The tree-independence number, introduced independently by  Yolov \cite{Yolov18} and  Dallard et al. \cite{DMS24}, retains the tree-decomposition framework but measures each bag by its independence number.
It is known that $tree\text{-}\alpha(G)\leq tw(G)+1$.
Unlike treewidth,  the tree-independence number can remain bounded even on dense graphs containing arbitrarily large cliques.
In particular, graphs of tree-independence number one are precisely the chordal graphs \cite{DMS24}, while chordal graphs may contain arbitrarily large cliques.
For example,  $tree\text{-}\alpha(K_n)=1$.
This parameter has consequently become the subject of an active structural
and algorithmic research program,
including the study of forbidden induced
subgraphs \cite{AACHSV24,CHLS26,DMS24III}, the relationship between
treewidth and clique number \cite{DMS24,DMS24III}, and algorithms for
finding and exploiting tree decompositions with bounded independence number
\cite{DFGKM26}.

Sim-width and induced matching treewidth weaken this form of control by using induced matchings, but they belong to two different decomposition frameworks.
Sim-width, introduced by Kang et al. \cite{KKST17}, is defined via branch decompositions. It bounds the size of an induced matching of the whole graph that crosses each cut induced by the decomposition.
Sim-width provides a complementary cut-based description of graph structure, and its relationship with other width parameters has been studied in several recent works \cite{BKR23,BMPY25}.
Induced matching treewidth originates in Yolov's \cite{Yolov18} minor-matching hypertree width, while the graph-theoretic terminology and notation used here follow Lima et al. \cite{LMMORS26}.
It minimizes, over all tree decompositions, the largest induced matching whose every edge has an end in a common bag. Despite being weaker than tree-independence number, bounded induced matching treewidth still retains significant algorithmic consequences \cite{BFK26,LMMORS26,Yolov18}.

For a graph $G$,
$$simw(G)\leq tree\text{-}\mu(G)\leq tree\text{-}\alpha(G).$$
where the first inequality is due to Bergougnoux et al. \cite{BKR23}, and the second follows directly from the definitions.
In contrast, tree-independence number cannot be bounded in terms of either sim-width or induced matching treewidth: $K_{n,n}$ has sim-width at most one and induced matching treewidth one, but tree-independence number $n$.
Large induced bicliques, that is, large induced complete bipartite graphs,  constitute a fundamental obstruction to translating induced-matching control into bounded independence number within the bags. This observation motivates recent work on tree-independence number in graphs excluding  induced bicliques \cite{ABCMMRW25,AMR25,BMPY25} and makes induced \(K_{t,t}\)-free graphs a natural setting for the problems studied in this paper.

The first problem concerns sim-width. Abrishami et al. {\cite{ABCMMRW25}}
proved that, after excluding a fixed induced biclique, bounded induced matching
treewidth implies bounded tree-independence number.
Alon et al. {\cite{AMR25}}  obtained the
polynomial estimate $tree\text{-}\alpha(G)=O_t\!\left(\mu^{3t^2+1}\right)$
for induced $K_{t,t}$-free graphs with $tree\text{-}\mu(G)\leq\mu$.
Since $simw(G)
\leq tree\text{-}\mu(G)$,
it is natural to ask whether induced matching treewidth in these results can be replaced by the potentially smaller parameter sim-width.
This problem was posed by Brettell et al. {\cite{BMPY25}}, and was later restated in an equivalent form by Abrishami et al. {\cite{ABCMMRW25}}.

\begin{problem}[Abrishami et al.\
{\cite{ABCMMRW25}}; Brettell et al.
{\cite{BMPY25}}]\label{prob:simwidth}
Do graph classes of bounded sim-width that exclude some fixed biclique as an
induced subgraph have bounded tree-independence number?
\end{problem}

Our first theorem gives a polynomial bound that answers Problem \ref{prob:simwidth}.

\begin{theorem}\label{thm:main}
For every integer $t\geq 2$, there exists a constant $c_t>0$ such
that, for every integer $s\geq 1$, every induced $K_{t,t}$-free
graph $G$ with $simw(G)\leq s$ satisfies
$tree\text{-}\alpha(G)
\leq c_t(s+1)^{2t^2-2t}.$
Equivalently,
$tree\text{-}\alpha(G)
=O_t\left(
\bigl(s+1\bigr)^{2t^2-2t}
\right)$.
\end{theorem}
The case $t=1$ is trivial: induced $K_{1,1}$-free graph is edgeless and has tree-independence number at most one.
Together with $simw(G)\leq tree\text{-}\mu(G)
\leq tree\text{-}\alpha(G)$,
Theorem \ref{thm:main} shows that, for every fixed integer $t\geq 2$, sim-width,
induced matching treewidth, and tree-independence number are polynomially
related on the class of induced $K_{t,t}$-free graphs.
Since $simw(G)\leq tree\text{-}\mu(G)$ and we have $tree\text{-}\alpha(G)=O_t\left(
\bigl(\mu+1\bigr)^{2t^2-2t}\right)=O_t\left(\mu^{2t^2-2t}\right)$, for every fixed $t$, Theorem \ref{thm:main} improves the exponent of induced matching treewidth obtained by  Alon et al. \cite[Theorem~1.2]{AMR25} from $3t^{2}+1$ to $2t^{2}-2t$.
Theorem \ref{thm:main} also gives a polynomial strengthening of a recent conjecture of Be\v{s}ter \v{S}torgel et al. \cite{BCKV26},  which asks whether $K_{1,d}$-free graph classes of bounded sim-width have bounded tree-independence number.
Indeed, every induced copy of $K_{t,t}$ contains an induced copy of $K_{1,t}$.
Hence, every induced $K_{1,t}$-free graph is induced $K_{t,t}$-free. Therefore, this conjecture is a direct consequence of Theorem \ref{thm:main}.

We now turn to the complementary case in which induced matching treewidth is fixed and the dependence on $t$ is considered.
Alon et al. \cite{AMR25} proved that there is no bivariate
polynomial $P$ such that $tree\text{-}\alpha(G) \leq P(tree\text{-}\mu(G),t)$ for every induced $K_{t,t}$-free graph $G$.
They asked whether polynomial dependence on $t$ is possible when induced matching treewidth is fixed.

\begin{problem}[Alon-Milani\v{c}-Rz\k{a}\.zewski
{\cite{AMR25}}]\label{prob:fixed-mu}
For every positive integer $\mu$, does there exist a polynomial
$p_\mu(t)$ such that, for every positive integer $t$, every
induced $K_{t,t}$-free graph $G$ with $tree\text{-}\mu(G)\leq\mu$ satisfies
$tree\text{-}\alpha(G)\leq p_\mu(t)$?
\end{problem}
Theorem \ref{thm:mu-main} answers Problem \ref{prob:fixed-mu} affirmatively.

\begin{theorem}\label{thm:mu-main}
For every integer $\mu \ge 1$, there exist positive constants
$C_{\mu}$ and $c_{\mu}$ such that, for every integer $t\ge 2$,
every induced $K_{t,t}$-free graph $G$ with
$tree\text{-}\mu(G)\le \mu$ satisfies
$tree\text{-}\alpha(G)
\le C_{\mu}t^{c_{\mu}}$.
Equivalently,
$tree\text{-}\alpha(G)=t^{O_{\mu}(1)}.$
\end{theorem}

For fixed $t$, Theorem~\ref{thm:main} gives a polynomial bound in $\operatorname{simw}(G)$.
Since $\operatorname{simw}(G)\le\operatorname{tree}\text{-}\mu(G)$,
it also gives the bound $c_t(\mu+1)^{2t^2-2t}$ when
$\operatorname{tree}\text{-}\mu(G)\le\mu$.
For fixed $\mu$, this bound contains a factor exponential
in $t^2$, whereas the bound $C_\mu t^{c_\mu}$ in
Theorem~\ref{thm:mu-main} is polynomial in $t$.
Both theorems  bound the tree-independence number, but their proofs are fundamentally different.
The proof of Theorem~\ref{thm:main} first uses crossing-edge covers to
bound the $\alpha$-order of every strong bramble, and then applies a result of Choi et al.~\cite{CHMW26} showing that a uniform upper bound on the $\alpha$-order of all strong brambles gives an upper bound on the tree-independence number.
The proof of Theorem~\ref{thm:mu-main} begins with a tree decomposition witnessing bounded induced matching treewidth and then transforms it into a new tree decomposition whose bags have bounded independence number.

The paper is organized as follows. Section~\ref{s2} introduces the notation and recalls the Helly property for subtrees, together with a result relating the tree-independence number to the $\alpha$-order of strong brambles.
Section~\ref{s3} proves the polynomial crossing-edge-cover theorem and applies it to  establish Theorem~\ref{thm:main}.
Section~\ref{s4} develops a matching-extraction argument under bounded VC-dimension and uses it  to prove Theorem~\ref{thm:mu-main}.
Section~\ref{s5} concludes the paper with some final remarks.

\section{Preliminaries}\label{s2}

All graphs considered in this paper are  finite and simple.  Let  $G$ be a graph.
 We denote by $\alpha(G)$ the independence number of $G$.
A \emph{matching} is a set of pairwise nonadjacent edges, and $\nu(G)$ denotes the maximum size of a matching in $G$.  An \emph{induced matching} in $G$ is a set $M\subseteq E(G)$ such that the subgraph induced by the ends of the edges in $M$ consists exactly of the edges of $M$.
Equivalently, $M$ is a matching  and there is no edge of $G$ joining ends of two distinct edges of $M$.
We write $im(G)$ for the maximum size of an induced matching in $G$.
For a set $X\subseteq V(G)$, let $\mu_{G}(X)$ denote the maximum size of an induced
matching $M$ in $G$ such that every edge of $M$ has at least one end in $X$.
We denote by $G[X]$ the subgraph of $G$ induced by $X$ and $N_G(X)$  the set
of vertices in $V(G)\setminus X$ having a neighbor in $X$, and set $N_G[X]=X\cup N_G(X)$.
If $X=\{v\}$, we write $N_G(v)$ and $N_G[v]$ for $N_G(X)$ and $N_G[X]$, respectively.
For a positive integer $k$, write $[k]=\{1,\ldots,k\}$.
For nonnegative integers $n$ and $k$, the binomial coefficient $\binom{n}{k}$ denotes the number of $k$-element subsets of an $n$-element set. In particular, $\binom{n}{k}=0$ if $k>n$.
For two graphs $G_1$ and $G_2$ on the same vertex set, $G_1\cup G_2$ denotes the graph with edge set $E(G_1)\cup E(G_2)$.

A \emph{tree decomposition} of $G$ is a pair $\mathcal{T}=(T,\beta)$, where $T$ is a tree
and $\beta\colon V(T)\to 2^{V(G)}$, satisfying:
\begin{enumerate}[label=(\roman*)]
  \item $\bigcup\limits_{x\in V(T)}\beta(x)=V(G)$;
  \item for every edge $uv\in E(G)$, there exists $x\in V(T)$ such that $\{u,v\}\subseteq \beta(x)$; 
  \item for every $v\in V(G)$, the set
  $\{x\in V(T):v\in\beta(x)\}$ induces a  subtree of $T$.
\end{enumerate}
Let $\mathcal{T}=(T,\beta)$ be  a tree decomposition of $G$.
For each $x\in V(T)$, the set $\beta(x)$ is called the \emph{bag} associated with $x$.
Define $\mu(\mathcal T)=\max\limits_{x\in V(T)}\mu_{G}(\beta(x))$.
The \emph{induced matching treewidth} of $G$, denoted by $tree\text{-}\mu(G)$, is defined as  $\min\limits_{\mathcal T}\mu(\mathcal T)$. 
The \emph{independence number} of $\mathcal{T}$, denoted by $\alpha(\mathcal T)$, is defined as $\max\limits_{x\in V(T)}\alpha(G[\beta(x)])$.
The \emph{tree-independence number} of  $G$, denoted by  \emph{tree$\text{-}\alpha$(G)}, is the minimum independence number of tree decompositions of $G$.

A \emph{branch decomposition} of $G$ is a pair $(T,\varphi)$, 
where $T$ is a tree in which every non-leaf vertex has degree three, and $\varphi:V(G)\to L(T)$ is a bijection from $V(G)$ to the leaf set $L(T)$ of $T$.
Every edge $e\in E(T)$ defines a cut $(A_e,B_e)$ of $G$: the two components of $T-e$
partition the leaves of $T$, and $A_e$ and $B_e$ are the corresponding vertex subsets of $G$.
For a cut $(A,B)$ of $G$, let  $E_G(A,B)=\{uv\in E(G):u\in A,\ v\in B\}$.
The \emph{sim-value} of a cut $(A,B)$ is defined by
$$sim_G(A,B):=\max\bigl\{|M|: M\subseteq E_G(A,B)\text{ is an induced matching in } G\bigr\}.$$
The \emph{sim-width} of a branch decomposition $(T,\varphi)$ is $\max\limits_{e\in E(T)} sim_G(A_e,B_e)$,
and the \emph{sim-width} of $G$, denoted by $simw(G)$, is the minimum sim-width over all branch decompositions of $G$.
For graphs with at most one vertex, we use the convention $simw(G)=0$.
For an edge set $F\subseteq E(G)$, a set $C\subseteq V(G)$ \emph{covers} $F$
if every edge in $F$ has at least one end in $C$.  In particular, a set
covering $E_G(A,B)$ is called a \emph{crossing-edge cover} of the cut $(A,B)$.

Let $\Omega$ be a finite set and let $\mathcal{F}\subseteq 2^\Omega$
be a set system. 
We call $S$ a \emph{transversal} of $\mathcal{F}$ if $|S\cap F|=1$ for every $F\in\mathcal{F}$. In addition,  if $S$ is an independent set in $G$, then we call $S$ an \emph{independent transversal} of $\mathcal{F}$.
For $S\subseteq \Omega$, 
write $\mathcal{F}|_S =\{F\cap S : F\in\mathcal{F}\}$.
We say that $S$ is \emph{shattered} by $\mathcal{F}$ if $\mathcal{F}|_S=2^S$;
equivalently, for every $X\subseteq S$, there exists an $F\in\mathcal{F}$ such that $F\cap S=X$.
The \emph{VC-dimension} of $\mathcal{F}$ is defined by
$$\operatorname{VC}(\mathcal{F}):= \max\bigl\{|S|: S\subseteq\Omega \text{ and } S \text{ is shattered by }\mathcal{F}
\bigr\}.$$
We use the convention that $\operatorname{VC}(\mathcal{F})=0$ if $\mathcal{F}$ shatters no nonempty subset of $\Omega$.
For a bipartite graph $G$ with bipartition $(U,V)$, let
$\mathcal{N}_G(U):=\{N_G(u):u\in U\}\subseteq 2^V$ and
$\mathcal{N}_G(V):=\{N_G(v):v\in V\}\subseteq 2^U$.
We define the VC-dimension of $G$ by
$$\operatorname{VC}(G):=\max\left\{\operatorname{VC}\bigl(\mathcal{N}_G(U)\bigr), \operatorname{VC}\bigl(\mathcal{N}_G(V)\bigr)\right\}.$$
Two distinct vertices in the same part of a bipartite graph are called
\emph{twins} if they have the same neighborhood in the opposite
part.

A \emph{strong bramble} in $G$ is a nonempty family $\mathcal{B}$ of pairwise intersecting nonempty vertex sets, each of which induces a connected subgraph of $G$.
A set
$X\subseteq V(G)$ \emph{covers} $\mathcal B$ if $X\cap B\neq\varnothing$ for
every $B\in\mathcal B$.  The \emph{$\alpha$-order} of $\mathcal B$ is
  $ord_{\alpha}(\mathcal B)
  :=\min\{\alpha(G[X]):X\subseteq V(G)\text{ covers }\mathcal B\}$.

The following lemma gives classical Helly property for subtrees of a tree.
\begin{lemma}[Helly property for subtrees \cite{Horn1972}]\label{HellyP}
A collection of subtrees of a tree has a common vertex if and only if
every two members of the collection have a common vertex.
\end{lemma}

We also use the following result of Choi et al. \cite{CHMW26}, which relates
tree-independence number to the $\alpha$-order of strong brambles.

\begin{lemma}[Choi et al. \cite{CHMW26}]
\label{thm:bramble-duality}
Let $G$ be a graph and let $k\geq 1$ be an integer.  If
$tree\text{-}\alpha(G)\geq 4k-2$,
then $G$ has a strong bramble of $\alpha$-order at least $k$.
\end{lemma}

\section{Proof of  Theorem~\ref{thm:main}
}\label{s3}

We first show that every cut of bounded sim-value admits a crossing-edge cover whose independence number is polynomially bounded by the sim-value.

\begin{proposition}\label{thm:cut-cover-intro}
Let $t\geq 2$ and $s\geq 0$ be integers. Let $G$ be an
induced $K_{t,t}$-free graph and $(A,B)$ be a cut of $G$
such that
${sim}_G(A,B)\leq s$.
Then $(A,B)$ has a crossing-edge cover $C$ satisfying
$$\alpha(G[C])=O_t((s+1)^{2t^2-2t}).$$
\end{proposition}

We use the following three lemmas to prove Proposition \ref{thm:cut-cover-intro}.

\begin{lemma}[Alon et al. \cite{AMR25}]\label{extractinducedmatching}
There exists a function $\mathcal{M}(s,t)=O_t(s^{t})$ for which the following holds. Every induced $K_{t,t}$-free bipartite graph that contains a matching of size at least $\mathcal{M}(s,t)$ contains an induced matching
of size $s+1$.
\end{lemma}

\begin{lemma}[Hajebi and Spirkl~\cite{HS26a}]\label{lem:fixed-two-graph}
Let $t$ and $c$ be positive integers. Let $H_1$ and $H_2$ be induced $K_{t,t}$-free graphs on the same vertex set $V$ such that $\alpha(H_1\cup H_2)<c$.
Then there exist \(V_1,V_2\subseteq V\) with \(V_1\cup V_2=V\) such that
$\alpha\bigl(H_i[V_i]\bigr)<(2t)^{t(c-1)}$ for each \(i\in\{1,2\}\).
\end{lemma}

Define the polynomial
$p_t(r)=r+(t-1)\binom rt+(2t)^{t(2t-2)}\sum_{j=1}^{2t-2}\binom rj.$
Then for fixed $t$, $p_t(r)=O_t((r+1)^{2t-2}).$

\begin{lemma}
\label{prop:poly-selector}
Let $t$ and $r$ be positive integers, $t\geq2$.  Let $H_1$ and $H_2$ be induced $K_{t,t}$-free graphs on the same vertex set $V$ satisfying $\alpha(H_1\cup H_2)\leq r.$
Then  there is a partition $\{V_1,V_2\}$ of $V$ such that
$$\alpha(H_1[V_{1}])\leq p_t(r), \ \  \alpha(H_2[V_{2}])\leq p_t(r).$$
\end{lemma}

\begin{proof}
Let $H=H_1\cup H_2$ and $I_H$ be a maximum independent set of $H$. Then $|I_H|\leq r$.  Define
\begin{align*}
  D_1&=\{v\in V\setminus I_H:|N_{H_1}(v)\cap I_H|\geq t\},\\
  D_2&=\{v\in V\setminus(I_H\cup D_1):
                 |N_{H_2}(v)\cap I_H|\geq t\}.
\end{align*}

We first claim that
  $$\alpha(H_1[D_1])\leq (t-1)\binom{|I_H|}{t}, \ \ \alpha(H_2[D_2])
  \leq (t-1)\binom{|I_H|}{t}.$$
If $|I_H|<t$, then no vertex has $t$ neighbors in $I_H$. Hence
$D_1=D_2=\varnothing$.
Since $\binom{|I_H|}{t}=0$, both inequalities hold.
Now suppose that $|I_H|\ge t$.
Let $I_{D_1}$ be the maximum independent set of $H_1[D_1]$.
For each $v\in I_{D_1}$, choose a $t$-set $I_v\subseteq N_{H_1}(v)\cap I_H$.
For any fixed $I\subseteq I_H$ with $|I|=t$, there are at most $t-1$ vertices $v\in I_{D_1}$ such that $I_v=I$.
Indeed, if there are $t$ such vertices, then these vertices together with $I$ induce a $K_{t,t}$ in $H_1$, a contradiction.
Moreover, there are at most $\binom{|I_H|}{t}$  subsets of size $t$ in $I_H$. Then $\alpha(H_1[D_1]) \leq (t-1)\binom{|I_H|}{t}$.
The second inequality follows similarly. The Claim follows.

Let $L=V\setminus(I_H\cup D_1\cup D_2)$ and $q=\sum_{k=1}^{2t-2}\binom {|I_H|} {k}$.
Note that for any vertex $v\in L$ and $i\in \{1,2\}$, $|N_{H_i}(v)\cap I_H|\leq t-1$.
Then $|N_{H}(v)\cap I_H|\leq 2t-2$.
Let $C_j\subseteq I_H$ with $|C_j|\le 2t-2$, $j\in \{1,2,\ldots, q\}$.
Note that for every vertex $v$ of $L$,  $N_H(v)\cap I_H \neq \emptyset$. Otherwise $I_H\cup \{v\}$ is a larger independent set of $H$, a contradiction.
For each  $j\in [q]$, define $L_{C_j}=\{v\in L:N_H(v)\cap I_H=C_j\}$ such that every vertex $v$ of $L$ belongs to a unique $L_{C_j}$.
Then $\{L_{C_1},\ldots,L_{C_q}\}$ form a partition of $L$.
It is possible that some of them are empty.
Let $I_{C_j}$ be an independent set in $H[L_{C_j}]$, then
$I_{C_j}\cup(I_H\setminus C_j)$ is independent in $H$.  Thus
$|I_{C_j}|+|I_{H}|-|C_j|\leq|I_{H}|$, that is $|I_{C_j}|\leq|C_j|$. This implies that
$$\alpha(H[L_{C_j}])\leq |C_j|\leq 2t-2.$$
Applying Lemma~\ref{lem:fixed-two-graph} to $H_1[L_{C_j}]$ and $H_2[L_{C_j}]$,  $j\in[q]$, we have a partition
$\{L_{C_j}^{'}, L_{C_j}^{''}\}$ of $L_{C_j}$
such that
$$ \alpha(H_1[L_{C_j}^{'}])\leq (2t)^{t(2t-2)} \mbox{ and } \alpha(H_2[L_{C_j}^{''}])\leq (2t)^{t(2t-2)}.$$

Let
\begin{align*}
  V_1=I_H\cup D_1\cup\bigcup_{j} L_{C_j}^{'},\quad
  V_2=D_2\cup\bigcup_{j} L_{C_j}^{''}.
\end{align*}
Then $\{V_1, V_2\}$ is a partition of $V$.  By subadditivity of the independence number, we have
\[
  \alpha(H_1[V_1])
  \leq |I_H|+(t-1)\binom {|I_H|} {t}
  +(2t)^{t(2t-2)}\sum_{j=1}^{2t-2}\binom {|I_H|} {j}
  \leq p_t(r).
\]
The same calculation, without the first term, bounds $\alpha(H_2[V_2])$ by
$p_t(r)$.  This proves the lemma.
\end{proof}

\begin{proof}[Proof of Proposition~\ref{thm:cut-cover-intro}]
Let $F=E_G(A,B)$.  If $F=\varnothing$, take $C=\varnothing$.  Otherwise,
write every $e\in F$ as $e=a_eb_e$, where $a_e\in A$ and $b_e\in B$.
Define auxiliary graphs $H_A$ and $H_B$ on vertex set $F$ and edge sets
$$E(H_A)=\{ef: a_e=a_f \text{ or } a_ea_f\in E(G)\},$$
and
$$E(H_B)=\{ef: b_e=b_f \text{ or } b_eb_f\in E(G)\}.$$

We first claim that $H_A$ and $H_B$ are induced $K_{t,t}$-free. Suppose, to the contrary, that $H_A$
contains an induced $K_{t,t}$ with bipartition $X=\{e_1,\ldots,e_t\}$ and $Y=\{f_1,\ldots,f_t\}.$
Since $X$ is independent in $H_A$, the vertices $a_{e_1},\ldots,a_{e_t}$
are distinct and form an independent set in $G$.
The same conclusion holds for the vertices $a_{f_1},\ldots,a_{f_t}$.
Suppose that there exists  $e_i\in X$ and  $f_j\in Y$ such that  $a_{e_i}=a_{f_j}$.
Because $t\geq2$, there is an $e_{i'}\in X\setminus\{e_i\}$.
Since  $X$ is independent in $H_A$, we have $a_{e_{i'}}\neq a_{e_i}$ and $a_{e_{i'}}a_{e_i}\notin E(G)$.
Since $a_{e_i}=a_{f_j}$, we have $a_{e_{i'}}\neq a_{f_j}$ and $a_{e_{i'}}a_{f_j}\notin E(G)$.
Thus $e_{i'}f_j\notin E(H_A)$, contradicting the fact that  $H_A[X\cup Y]\cong K_{t,t}$.
Hence, the vertices $a_{e_1},\ldots,a_{e_t},a_{f_1},\ldots,a_{f_t}$ are distinct.
Since $e_if_j\in E(H_A)$, we have $a_{e_i}a_{f_j}\in E(G)$ for every  $i,j\in[t]$. Therefore,
$G\bigl[\{a_{e_1},\ldots,a_{e_t}\} \cup \{a_{f_1},\ldots,a_{f_t}\}\bigr]\cong K_{t,t}$,
contradicting the assumption  that $G$ is
induced $K_{t,t}$-free. Hence $H_A$ is
induced $K_{t,t}$-free. By symmetry, the same holds for $H_B$.

We next claim that $\alpha(H_A\cup H_B)\leq \mathcal{M}(s,t)-1$, where $\mathcal{M}(s,t)$ is given by Lemma \ref{extractinducedmatching}.
Let $I$ be an independent set of $H_A\cup H_B$. Here
$I\neq\varnothing$, since the desired bound is trivial otherwise.  Then the sets $A_I=\{a_e:e\in I\}$ and $B_I=\{b_e:e\in I\}$ are independent sets of $G$, both of size $|I|$, and the edges indexed by
$I$ form a matching of $G$.  The graph $G[A_I\cup B_I]$ is bipartite and has no $K_{t,t}$ as subgraph.
If $|I|\geq \mathcal{M}(s,t)$, by lemma
\ref{extractinducedmatching}, $G[A_I\cup B_I]$ contains an induced matching
of size $s+1$, contradicting the fact that $sim_G(A, B)\leq s$.  Hence
$\alpha(H_A\cup H_B)\leq \mathcal{M}(s,t)-1$.

By Lemma~\ref{prop:poly-selector}, we can obtain a partition
$\{F_1, F_2\}$ of $F$ such that
$\alpha(H_A[F_1])\leq p_t(\mathcal{M}(s,t)-1)$ and $\alpha(H_B[F_2])\leq p_t(\mathcal{M}(s,t)-1)$. Let $$C=\{a_e:e\in F_1\}\cup\{b_e:e\in F_2\}.$$
This set covers every edge of $F$.  Let $I_1$ be an independent set in
$G[\{a_e:e\in F_1\}]$.
The edges in $F_1$ that are indexed by $I_1$ form
an independent set in $H_A[F_1]$, so $|I_1|\leq \alpha(H_A[F_1])\leq p_t(\mathcal{M}(s,t)-1)$.  The analogous statement holds on $H_B[F_2]$.  Since $A\cap B=\varnothing$, we conclude that $\alpha(G[C])  \leq2p_t(\mathcal{M}(s,t)-1)$, and so $\alpha(G[C])=O_t((s+1)^{2t^2-2t}).$
This proves the proposition.
\end{proof}

We next transfer cut covers in a branch decomposition to a tree decomposition.
\begin{proposition}\label{prop:transfer-intro}
Let $(T,\varphi)$ be a branch decomposition of a graph $G$ and $q$ be a positive
integer.  Suppose that for
each edge $e\in E(T)$, $(A_e,B_e)$ have a
crossing-edge cover $C_e$ satisfying $\alpha(G[C_e])\leq q$. Then $tree\text{-}\alpha(G)\leq 8q+1$.
\end{proposition}

\begin{proof}
Let $\mathcal B$ be a strong bramble in $G$. Recall that for a vertex $v\in V(G)$, $\varphi(v)$ is a leaf of $T$.
For each $B\in\mathcal B$, let $T_B$ be the minimal subtree of $T$ containing $\{\varphi(v):v\in B\}$.
 Since any two members $B$ and $B'$ in  $\mathcal B$ intersect,
the subtrees $T_B$ and $T_{B'}$ share a leaf.  Hence the family
$\{T_B:B\in\mathcal B\}$ is pairwise intersecting.  By Lemma \ref{HellyP}, all trees of $\{T_B:B\in\mathcal B\}$ have a vertex $z$ in common, that is, $z\in\bigcap_{B\in\mathcal B}V(T_B)$.

If $z$ is a leaf of $T$,  write $z=\varphi(v)$.
We  have $z\in\varphi(B)$. Otherwise, deleting $z$ from $T_B$ would leave a smaller subtree containing $\varphi(B)$,
contrary to the minimality of $T_B$.
Thus $v\in B$ for every $B\in\mathcal B$, and $\{v\}$ covers $\mathcal B$.  Therefore
$ord_{\alpha}(\mathcal B)\leq1\leq2q.$

If $z$ is not a leaf  of $T$, then the degree of $z$ in $T$ is three. Let $e_1,e_2,e_3$ be the three edges in $T$ incident with $z$.
The components of $T-z$ define a partition $\{V_1, V_2, V_3\}$ of $V(G)$, where $V_i$ is the set of vertices whose corresponding leaves lie in the component that contains an end of $e_i$.
For every $B\in\mathcal B$, since $z\in T_B$ and $T_B$ is minimal,  $B$ meets at least two of $V_1,V_2,V_3$.

We claim that $C_{e_1}\cup C_{e_2}$ covers $\mathcal B$.  Fix
$B\in\mathcal B$.  If $B\cap V_1\neq\varnothing$, then $B$ also meets
$V_2\cup V_3$.  Since $G[B]$ is connected, it contains an edge of
$E_G(V_1,V_2\cup V_3)$.  This edge is covered by $C_{e_1}$, and both its
ends belong to $B$. Hence $B\cap C_{e_1}\neq\varnothing$.
If $B\cap V_1=\varnothing$, then $B$ meets both $V_2$ and $V_3$.
The connectivity of  $G[B]$ gives an edge of $E(G[B])\cap E_G(V_2,V_1\cup V_3)$, and therefore $B\cap C_{e_2}\neq\varnothing$.  This proves the claim.

For every independent set $I\subseteq C_{e_1}\cup C_{e_2}$, we have
$$|I| \leq |I\cap C_{e_1}|+|I\cap C_{e_2}| \leq \alpha(G[C_{e_1}])+\alpha(G[C_{e_2}])  \leq2q.$$
Consequently,
$$ord_{\alpha}(\mathcal B)\leq \alpha\bigl(G[C_{e_1}\cup C_{e_2}]\bigr)\leq2q.$$

We have proved that every strong bramble in $G$ has $\alpha$-order at most $2q$.
By Lemma~\ref{thm:bramble-duality},  we have $ tree\text{-}\alpha(G)\leq8q+1$.
\end{proof}


\begin{proof}[Proof of Theorem~\ref{thm:main}]
The case $|V(G)|\leq1$ is immediate. Since $simw(G)\leq s$, we have a branch
decomposition $(T,\varphi)$ of $G$ of sim-width at most $s$.  For every $e\in E(T)$,
the corresponding cut $(A_e,B_e)$ satisfies
$sim_G(A_e,B_e)\leq s$.
By Proposition~\ref{thm:cut-cover-intro}, there is a crossing-edge cover $C_e$ such that $\alpha(G[C_e])=O_t((s+1)^{2t^2-2t}).$
By Proposition~\ref{prop:transfer-intro},
$$tree\text{-}\alpha(G)\leq c_t(s+1)^{2t^2-2t},$$
where $c_t>0$ is a constant depending only on $t$.
The proof is complete.
\end{proof}

\section{Proof of Theorem  \ref{thm:mu-main}
}
\label{s4}


The following lemma gives the relation between degeneracy and chromatic number, which is used in the proof of Theorem  \ref{thm:mu-main}.
A graph $G$ is called \emph{$k$-degenerate} if every nonempty induced subgraph $H$ of $G$ contains a vertex $v$ with degree at most $k$ in $H$. We use $\delta(H)$ to denote the minimum degree of $H$.

\begin{lemma}[Szekeres--Wilf \cite{SW68}]
\label{lem:Szekeres-Wilf}
For every graph $G$,
$$\chi(G) \leq 1+\max\bigl\{\delta(H): H\text{ is a nonempty induced subgraph of }G\bigr\}.$$
In particular, every $k$-degenerate graph is $(k+1)$-colorable.
\end{lemma}

 We next show that bounded induced matching number implies bounded VC-dimension.

\begin{lemma}\label{lem:vc-from-im}
Let $\mu\geq1$ be an integer. If $G$ is a bipartite graph with $im(G)\leq \mu$, then
$\operatorname{VC}(G)\leq \mu.$
\end{lemma}

\begin{proof}
Let $G$ be a bipartite graph with bipartition $(X,Y)$.  Suppose that there exists a set $Y'=\{y_1,\ldots,y_{\mu+1}\}\subseteq Y$ shattered by $\{N_G(x):x\in X\}$.
For every $i$, choose $x_i\in X$ such that $N_G(x_i)\cap Y'=\{y_i\}.$
Then $\{x_1,x_2,\ldots,x_{\mu+1}\}$ and $\{y_1,y_2,\ldots,y_{\mu+1}\}$ are independent set in $G$ and $G[\{x_1,x_2,\ldots,x_{\mu+1}\}\cup\{y_1,y_2,\ldots,y_{\mu+1}\}]$ only contains edges $\{x_1y_1, x_2y_2,\ldots,x_{\mu+1}y_{\mu+1}\}$. These edges form an induced
matching of size $\mu+1$, a contradiction.  Thus $\mathcal{N}_G(X)$ has
VC-dimension at most $\mu$.  The same bound holds for $\mathcal{N}_G(Y)$. Therefore, $\operatorname{VC}(G)\leq \mu$.
\end{proof}

The following threshold form of  Hons's theorem is used to obtain a matching bound that is polynomial in $t$.

\begin{lemma}[Hons \cite{Hons26}]
\label{thm:hons}
For every integer $\mu\geq1$, there exist a positive integer $N_{\mu}$ and two positive constants $A_{\mu}$ and $\beta_{\mu}$ such that the following holds.
Let $G$ be a bipartite graph with bipartition $(X,Y)$ and $\operatorname{VC}(G)\leq {\mu}$, where $Y$ has at least $N_{\mu}$ vertices and no twins. Then $G$ contains an induced subgraph $G^{'}=(X^{'},Y^{'})$ with $|X^{'}|=|Y^{'}|\geq A_{\mu}|Y|^{\beta_{\mu}}$ that is isomorphic to either a matching, a co-matching, or a half-graph.
\end{lemma}

Here, a \emph{co-matching} is  a bipartite graph with vertex set $\{x_1,\ldots,x_k\}\cup \{y_1,\ldots,y_k\}$ and edge set $\{x_iy_j:  i, j \in [k] \mbox{ with } i\neq j\}$.
A \emph{half-graph} is a graph with the same vertex set and edge set $\{x_iy_j:  i, j \in [k] \mbox{ with } i\leq j\}$.
Note that a co-matching contains an induced $K_{t,t}$ on $\{x_1,\ldots,x_t\}$ and $\{y_{t+1},\ldots,y_{2t}\}$, and a half-graph contains one on $\{x_1,\ldots,x_t\}$ and $\{y_t,\ldots,y_{2t-1}\}$.


\begin{lemma}\label{lem:vc-matching}
For every integer $\mu\geq1$, there exist a positive integer $N_{\mu}$ and two positive constants $A_{\mu}$ and $\beta_{\mu}$ such that the following holds.
If $G$ is an induced $K_{t,t}$-free bipartite graph with $im(G)\leq\mu$, then $\nu(G)\leq\kappa_\mu(t),$
where $t\geq2$ and $\kappa_\mu(t)=(t-1)\bigl(\max
\left\{N_\mu,\left\lceil\left(\frac{ \max\{\mu+1,2t\}}{A_\mu}\right)^{1/\beta_\mu}\right\rceil
\right\}-1\bigr)$.
Moreover, 
$\kappa_\mu(t) =O_\mu\!\left(t^{1+1/\beta_\mu}\right).$
\end{lemma}


\begin{proof}
Choose $N_\mu,A_\mu,\beta_\mu$ as in Lemma~\ref{thm:hons}.
Suppose, to the contrary, that $G$ has a matching $M$ of size $|M|\geq\kappa_\mu(t)+1$.
 Let $(X,Y)$ be the bipartition of $G$, set $X_M=V(M)\cap X$, $Y_M=V(M)\cap Y$ and
$G_M=G[X_M\cup Y_M]$. Then $G_M$ is $K_{t,t}$-free, and $\operatorname{VC}(G_M)\leq\mu$ by Lemma~\ref{lem:vc-from-im}.

Partition $X_M$ into twin classes according to neighborhoods in $Y_M$.
Each twin class in $X_M$ contains at most $t-1$ vertices. Otherwise, if $x_1,\ldots,x_t$ belong to one twin
class, since  each $x_i\in X_M$ has a neighbor $y_i\in Y_M$, $|N_{Y_M}(x_i)|\geq t$. There is an induced $K_{t,t}$ in $G_M$, a contradiction.
Choose one vertex from each twin class and denote the resulting set by $X'$.
Since
 $|M|\geq\kappa_\mu(t)+1$,
we have $|X'|\geq\max\left\{N_\mu, \left\lceil \left(\frac{\max\{\mu+1,2t\}}{A_\mu}\right)^{1/\beta_\mu}\right\rceil \right\}.$

Let $G'=G[X'\cup Y_M]$.
Then the set $X'$ is pairwise non-twin in $G'$.
Since deleting vertices does not increase  the VC-dimension, we have $\operatorname{VC}(G')\leq\mu$.
By Lemma~\ref{thm:hons},  $G'$ contains an induced matching, a co-matching, or a half-graph,  with at least $\max\{\mu+1,2t\}$ vertices in each part.
If $G'$ contains an induced matching of size $\max\{\mu+1,2t\}$, then we get a contradiction to the assumption that $im(G)\leq\mu$.
If $G'$ contains a co-matching or half-graph with $\max\{\mu+1,2t\}$ vertices in each part, then $G'$ contains an induced $K_{t,t}$, a contradiction to the assumption that $G$ is induced $K_{t,t}$-free. Thus $\nu(G)\leq\kappa_\mu(t)$.
\end{proof}

\begin{lemma}\label{lem:independent-transversal}
Let $m\geq1$ and $t\geq2$ be integers.
Let $G$ be an  induced $K_{t,t}$-free graph and $\mathcal{I}=\{I_1,\ldots,I_m\}$ be a collection of $m$ pairwise disjoint independent sets in $G$ satisfying $|I_i|\geq t^{m-1}$ for every $i\in[m]$.
Then  $\mathcal{I}$ has an independent transversal in $G$.
\end{lemma}

\begin{proof}
We prove the lemma by induction on $m$.  The case $m=1$ is immediate.  Suppose that $m\geq2$ and  the lemma holds for $m-1$. If there is a  vertex $v\in I_m$ such that  $|I_i\setminus N_G(v)|\geq t^{m-2}$ for each $i\in[m-1]$, then applying the induction hypothesis to the $m-1$ independent sets $I_i\setminus N_G(v)$, we obtain an independent transversal of size $m-1$. Adding $v$ to this transversal yields an independent transversal of $\mathcal{I}$.

Now suppose that for each vertex $v\in I_m$,  there is an index $i\in[m-1]$ such that $|I_i\setminus N_G(v)|<t^{m-2}.$
We regard $i$ as the color assigned to $v$.
If some color $i\in[m-1]$ is assigned to at least $t$ vertices, let $X\subseteq I_m$ be a set of $t$ vertices of color $i$.
Then for every $v\in X$,  we have $|I_i\setminus N_G(v)|\leq t^{m-2}-1.$
Thus the number of common neighbors of $X$ in $I_i$ is at least
$$|I_i|-\sum_{v\in X}|I_i\setminus N_G(v)|
 \geq t^{m-1}-t(t^{m-2}-1)=t.$$
Since $I_i$ and $I_m$ are independent, $X$  together with its $t$ common neighbors in $I_i$  induce a $K_{t,t}$,
a contradiction. Thus every color class of $I_m$ has size at most $t-1$.
  Since $m-1\leq2^{m-2}\leq t^{m-2}$, we have $$|I_m|\leq(m-1)(t-1)<t^{m-1}.$$  This
contradiction completes the proof.
\end{proof}

We prove Theorem~\ref{thm:mu-main} by adapting the light--heavy tree-decomposition argument from \cite[Theorem~1.1, Claims~3.3--3.7]{ABCMMRW25}.  In our setting, Lemma~\ref{lem:vc-matching} plays the role of the matching-extraction lemma \cite[Lemma~3.1]{ABCMMRW25}.  
The main new idea is the construction of an auxiliary graph
on the heavy vertices of a fixed maximum independent set
that lie in a common bag. By bounding its chromatic number
and applying Lemma~\ref{lem:independent-transversal}, we obtain, for fixed $\mu$, a polynomial
bound in $t$ on the number of these vertices. This replaces
the Ramsey-type argument used in
\cite[Claim~3.5]{ABCMMRW25}.
More precisely, we show that every bag contains at most $\mu\bigl(2\kappa_\mu(t)+1\bigr)$ heavy vertices, which is polynomial in $t$ for each fixed $\mu$.  With this estimate in hand, the remaining steps follow the argument of \cite{ABCMMRW25}, with the corresponding bounds replaced by the estimates established in this section.

\begin{proof}[Proof of Theorem~\ref{thm:mu-main}]
 Recall that $\kappa_\mu(t)$ is defined in Lemma~\ref{lem:vc-matching}.
Choose a tree decomposition $\mathcal T=(T,\beta)$ of $G$ such that $\mu(\mathcal T)\leq \mu$, and let $S$ be a maximum independent set of $G$.

\textbf{Claim 1.}
For each $x\in V(T)$, $\alpha\bigl(G[\beta(x)\setminus S]\bigr)\leq \kappa_\mu(t)$.

The proof is the same as that of \cite[Claim~3.3]{ABCMMRW25}, with $\mathsf{M}(\mu,t)$ replaced by $\kappa_\mu(t)$ and with \cite[Lemma~3.1]{ABCMMRW25} replaced by Lemma~\ref{lem:vc-matching}.
For completeness, we give the details of the proof.

Let $I$ be a maximum independent set of $G[\beta(x)\setminus S]$.
For every $X\subseteq I$, 
$|N_G(X)\cap S|\geq |X|$,
otherwise $(S\setminus N_G(X))\cup X$ would be an independent set
with a greater size than $S$. Let $B=G[S\cup I]$. By Hall's theorem, $B$ has a matching
covering $I$. 
Note that every induced matching of $B$ is an induced matching of
$G$, and each one of its edges meets $I\subseteq\beta(x)$.  
Then
$im(B)\leq\mu_G(\beta(x))\leq\mu$.
By Lemma~\ref{lem:vc-matching}, we have $\nu(B)\leq \kappa_\mu(t)$. Therefore,
$|I|\leq \kappa_\mu(t)$.  
\hfill\qed

Let $C(\mu,t)=t^{\mu}+\mu\kappa_\mu(t)$.
A vertex $v$ of $G$ is called \emph{light} if $\alpha(G[N_G(v)])< C(\mu,t)$,
and \emph{heavy} otherwise. Let $S_\ell$ and $S_h$ denote the sets of light and heavy vertices of $S$, respectively.

\textbf{Claim 2.}
For each $x\in V(T)$, $\alpha\bigl(G[N_G(\beta(x)\cap S_\ell)]\bigr)\leq \mu \cdot(C(\mu,t)-1).$

Let $I_l$ be a maximum independent set of $G[N_G(\beta(x)\cap S_\ell)]$.
We follow the proof of \cite[Claim~3.4]{ABCMMRW25}, and include the details for completeness.
Choose a set $U\subseteq\beta(x)\cap S_\ell$ that is minimal under inclusion among all sets satisfying $I_l\subseteq N_G(U)$.
Then  $|I_l|\leq\sum_{u\in U}|I_l\cap N_G(u)|$ and $U$ is an independent set.

We assert that $|U|\leq\mu$.  Indeed, by the minimality of $U$, for each $u\in U$ there is a vertex $u'\in I_l\cap N_G(u)$ such that $u'\notin N_G(U\setminus{u})$.
Hence, these edges $uu'$ with $u\in U$ form an induced matching.
Since each such edge has an end in $\beta(x)$ and $\mu(\mathcal T)\leq\mu$, we have $|U|\leq\mu$.

Since every $u\in U$ is light and $C(\mu,t)$ is an integer, $|I_l\cap N_G(u)|\leq C(\mu,t)-1$.  Therefore, $|I_l|\leq\sum_{u\in U}|I_l\cap N_G(u)|\leq |U|\bigl(C(\mu,t)-1\bigr)\leq\mu\cdot\bigl(C(\mu,t)-1\bigr)$.
\hfill\qed

\textbf{Claim 3.}
For each $x\in V(T)$,
$|\beta(x)\cap S_h|\leq\mu\cdot(2\kappa_\mu(t)+1)$.

To prove Claim 3, let $x\in V(T)$.  For each $a\in \beta(x)\cap S_h$, $\alpha(G[N_G(a)])\geq C(\mu,t)$, and so we can choose an independent set $I_a\subseteq N_G(a)$ such that $|I_a|=C(\mu,t)$.
For distinct $a,b\in \beta(x)\cap S_h$, call the ordered pair $(a,b)$ \emph{bad} if $|N_G(b)\cap I_a|\geq \kappa_\mu(t)+1.$

We first show
that, for each $a\in \beta(x)\cap S_h$, there are at most $\kappa_\mu(t)$
vertices $b\in (\beta(x)\cap S_h)\setminus\{a\}$ such that $(a,b)$ is bad.  Suppose, to the contrary,
that there exists a set $B=\{b_1,\ldots,b_{\kappa_\mu(t)+1}\}$ of distinct vertices such that $(a,b_i)$ is bad for each $i\in[\kappa_\mu(t)+1]$.
Since every $b_i$ has at least $\kappa_\mu(t)+1$
neighbors in $I_a$, we can choose distinct vertices $u_i\in N_G(b_i)\cap I_a$, $i\in[\kappa_\mu(t)+1]$.
Then the graph $G_a:=G\bigl[I_a\cup B\bigr]$
contains a matching $\{b_iu_i:i\in[\kappa_\mu(t)+1]\}$ of size $\kappa_\mu(t)+1$.
Note that $I_a$ and $B$ are disjoint independent vertex sets. So $G_a$ is an induced $K_{t,t}$-free bipartite graph.
Since every edge of $G_a$ meets $B\subseteq\beta(x)$, we have $im(G_a)\leq\mu_G(\beta(x))\leq\mu$.
By Lemma \ref{lem:vc-matching}, $\nu(G_a)\leq \kappa_\mu(t)$, a contradiction.

Define an auxiliary graph $F$ with vertex set $\beta(x)\cap S_h$ and edge set
$\{ab: a, b\in \beta(x)\cap S_h, (a,b) \mbox{ or } (b,a) \mbox{ is bad}\}$.
For every $W\subseteq\beta(x)\cap S_h$, each edge $ab\in E(F[W])$ corresponds to at least one bad ordered pair, either $(a,b)$ or $(b,a)$.
Since each vertex appears as the first entry of at most $\kappa_\mu(t)$ bad ordered pairs, we have
$|E(F[W])|\leq \kappa_\mu(t)|W|$.
This implies that every nonempty induced subgraph of $F$ has average degree at
most $2\kappa_\mu(t)$, and therefore contains a vertex of degree at most $2\kappa_\mu(t)$.
Thus $F$ is $2\kappa_\mu(t)$-degenerate. By Lemma \ref{lem:Szekeres-Wilf},
$F$ is $(2\kappa_\mu(t)+1)$-colorable.

Suppose, to the contrary, that  $|\beta(x)\cap S_h|>\mu\cdot(2\kappa_\mu(t)+1)$.
Since every color class of $F$ is independent, then $\alpha(F)\geq \left\lceil\frac{|\beta(x)\cap S_h|}{2\kappa_\mu(t)+1}\right\rceil\geq \mu+1$.
Let $I_F=\{a_1,a_2,\ldots,a_{\mu+1}\}$ be an independent set in $F$.
Then no ordered pair of distinct vertices of $I_F$ is bad.  So for any distinct $a_i,a_j\in I_F$,  $|I_{a_i}\cap N_G(a_j)|\leq \kappa_\mu(t)$.
For each $a_i\in I_F$, let $I_{a_i}'=I_{a_i}\setminus \bigcup\limits_{a_j\in I_F\setminus\{{a_i}\}}N_G(a_j).$
Since $|I_F\setminus\{{a_i}\}|=\mu$ and $|I_{a_i}|=C(\mu,t)$, we have
$$|I_{a_i}'|\geq C(\mu,t)-\sum\limits_{a_j\in I_F\setminus\{{a_i}\}}|I_{a_i}\cap N_G(a_j)|\geq C(\mu,t)-\mu \kappa_\mu(t)=t^{\mu}.$$
For each $v\in I_{a_i}'$, since $v\in I_{a_i}\subseteq N_G(a_i)$, while $v\notin N_G(a_j)$ for every $a_j\in I_F\setminus\{a_i\}$, we have $N_G(v)\cap I_F=\{a_i\}$.
This implies that for distinct ${a_i},a_j \in I_F$,  $I_{a_i}'\cap I_{a_j}'=\emptyset$.
By Lemma \ref{lem:independent-transversal},
$\{I_{a_1}', I_{a_2}',\ldots, I_{a_{\mu+1}}^{'}\}$ admit an independent transversal $I':=\{u_{i}:i\in [\mu+1]\}$, where $u_{i}\in I_{a_i}'$.
Since $N_G(u_i)\cap I_F=\{a_i\}$ for every $i\in [\mu+1]$, and $I_F$ and $I'$ are independent sets, we see that $\{a_iu_{i}:i\in [\mu+1]\}$ is an induced matching in $G$ of size $\mu+1$, of which
each edge meets $I_F\subseteq\beta(x)$, contradicting $\mu_G(\beta(x))\leq\mu$.
Claim~3 follows.
\hfill\qed

Recall that $\mathcal{T}=(T,\beta)$ is a tree decomposition of $G$.
For a vertex $v$ of $G$, let $T_v$ be the subgraph of $T$ induced by the vertices that contain $v$ in their bags. Then $T_v$ is a nonempty tree.
We now construct a new tree decomposition $\mathcal{T}'=(T',\beta')$ of $G$ as follows:
\begin{itemize}
    \item The tree $T'$ is obtained from $T$ by adding, for every
    $s\in S_\ell$, a new leaf node $y_s$ adjacent to some node
    $x_s$ of $T_s$.

    \item For every node $x$ of $T$, we set
        $\beta'(x)
        =
        \bigl(\beta(x)\setminus S_\ell\bigr)
        \cup N_G\bigl(\beta(x)\cap S_\ell\bigr).$

    \item For every vertex $s\in S_\ell$, we set $\beta'(y_s)=N_G[s]$.
\end{itemize}

\textbf{Claim 4.}   \cite[Claim 3.6 ]{ABCMMRW25}
$\mathcal{T}'$ is a tree decomposition of $G$.

Next, using the above claims and an argument analogous to that in
\cite[Claim~3.7]{ABCMMRW25}, we can obtain the following bound for  $\mathcal T'$.

\textbf{Claim 5.}
$\alpha(\mathcal{T}')\leq(\mu+1)^2 \kappa_\mu(t)+\mu t^{\mu}$.

Following the proof of \cite[Claim~3.7]{ABCMMRW25} and replacing the corresponding bounds therein with those established in Claims~1-3, we obtain the following estimate.

For each vertex $x\in V(T)$, Claims~1-3 and the fact that  $C(\mu,t)=\mu\kappa_\mu(t)+t^\mu$ imply that
$$\alpha(G[\beta'(x)])\leq \kappa_\mu(t)+\mu\bigl(2\kappa_\mu(t)+1\bigr)+\mu\bigl(C(\mu,t)-1\bigr)\leq(\mu+1)^2\kappa_\mu(t)+\mu t^\mu.$$
 For each new leaf $y_s$ of $T'$, where $s\in S_\ell$, we have $\beta'(y_s)=N_G[s]$ and $\alpha(G[N_G(s)])\leq C(\mu,t)-1$. Hence
$\alpha(G[\beta'(y_s)])\leq C(\mu,t)-1\leq(\mu+1)^2\kappa_\mu(t)+\mu t^\mu$.
Claim 5 follows.
\hfill\qed

By definition of tree-independence number, $tree\text{-}\alpha(G) \leq \alpha(\mathcal T')\leq (\mu+1)^2\kappa_\mu(t)+\mu t^\mu.$
Recall that $\kappa_\mu(t)=O_\mu\!\left(t^{1+1/\beta_\mu}\right)$.
Thus there exist positive constants
$C_{\mu}$ and $c_{\mu}$ such that $tree\text{-}\alpha(G)\le C_{\mu}t^{c_{\mu}}$, and so  $tree\text{-}\alpha(G)=t^{O_\mu(1)}$. This completes the proof.
\end{proof}

\section{Concluding remarks}
\label{s5}

Our results give two polynomial bounds on the
tree-independence number of graphs excluding an induced $K_{t,t}$.
Theorem~\ref{thm:main} gives a polynomial answer to
Problem~\ref{prob:simwidth}: for fixed $t$, the tree-independence number of
an induced $K_{t,t}$-free graph is polynomially bounded in its sim-width.
Theorem~\ref{thm:mu-main} answers Problem~\ref{prob:fixed-mu}: for fixed induced matching treewidth, the tree-independence number is polynomially bounded in $t$.
In particular, Theorem~\ref{thm:main} converts information associated with the cuts of a branch decomposition into a bound on the independence number of the bags of a tree decomposition.
Next, we give the following consequence for unbalanced $K_{a,b}$.

\begin{corollary}\label{cor:unbalanced}
Let $s\geq1$ be an integer, let $a,b\geq2$ be positive integers and $m=\max\{a,b\}$.
If $G$ is an induced $K_{a,b}$-free graph with $\operatorname{simw}(G)\leq s$, then
$tree\text{-}\alpha(G)=O_m\left(\bigl(s+1\bigr)^{2m^2-2m}\right)$.
\end{corollary}

\begin{proof}
Since every induced $K_{m,m}$ contains an induced $K_{a,b}$, every induced $K_{a,b}$-free graph
is also induced $K_{m,m}$-free. By Theorem~\ref{thm:main} with $t=m$, we obtain
$tree\text{-}\alpha(G)=O_m\left(\bigl(s+1\bigr)^{2m^2-2m}\right)$.
\end{proof}

Alon et al. \cite[Lemma~5.1]{AMR25} showed that, for  induced $K_{t,t}$-free graphs $G$, the tree-independence number cannot be bounded by any bivariate polynomial in $tree\text{-}\mu(G)$ and $t$. We observe that the same construction shows that the tree-independence number cannot be bounded by any bivariate polynomial in $simw(G)$ and $t$ either.
In fact, for every positive integer $t$, their construction yields an induced $K_{t,t}$-free graph $G_t$ such that $im(G_t)\leq t-1$ and $tree\text{-}\alpha(G_t)=2^{\Omega(t)}$. Since $simw(G_t)\leq tree\text{-}\mu(G_t)\leq im(G_t)\leq t-1$, for every bivariate polynomial $P$ of total degree $d$, we have $\lvert P(simw(G_t),t)\rvert=O_P(t^d)$. As $tree\text{-}\alpha(G_t)=2^{\Omega(t)}$, we have  $tree\text{-}\alpha(G_t)>P(simw(G_t),t)$ for  sufficiently large $t$.
Hence, the tree-independence number of induced $K_{t,t}$-free graphs cannot be bounded by any bivariate polynomial in $simw(G)$ and $t$.

\vspace{0.3cm}
\noindent
\textbf{Acknowledgments.} 
This work is  supported by the National Natural Science Foundation of China under grant numbers 12571381 and 12371361. Mengyuan Niu is supported by the China Scholarship Council (No. 202507040066) and the Institute for Basic Science (IBS-R029-C4).
ChatGPT 5.4 was used during the exploratory stage of this project. All mathematical arguments and proofs in the manuscript were developed and rigorously verified by the authors.

\end{document}